\documentclass[journal]{ieeeconf}
\usepackage{booktabs}
\usepackage{lineno}
\usepackage[cmex10]{amsmath}
\usepackage{amscd,amsfonts,amsmath,amssymb,mathrsfs,pifont,stmaryrd,tipa}
\usepackage{array,cases,dsfont,graphicx,texdraw}
\usepackage{wrapfig}
\usepackage{graphics} % for pdf, bitmapped graphics files
\usepackage{epsfig} % for postscript graphics files
\usepackage{stfloats}
\usepackage{subfig}
\usepackage{hyperref}
\usepackage{color}
\newtheorem{Remark}{Remark}

\newtheorem{Problem}{Problem}

\newtheorem{Theorem}{Theorem}
\newtheorem{Lemma}{Lemma}

\newtheorem{Assumption}{Assumption}

\input{tcilatex}                                                          % paper

\IEEEoverridecommandlockouts                              % This command is only
\title{\LARGE \bf
Nash Equilibrium Seeking for Games in Second-order Systems without Velocity Measurement}

\author{Maojiao Ye, Jizhao Yin and Le Yin
\thanks{M. Ye and J. Yin are with the School of Automation, Nanjing University of Science and Technology, Nanjing 210094, P.R. China (Email: ye0003ao@e.ntu.edu.sg and yinjizhao@njust.edu.cn); L. Yin is with the School of Computer and Information Science, Southwest
University, Chongqing, China (Email: yinle0002@swu.edu.cn).}
\thanks{This work is supported by the Natural Science Foundation of China (NSFC) under Grant 61803202 and the Natural Science Foundation of Jiangsu Province, No. BK20180455.}
}
\begin{document}

\maketitle
\thispagestyle{empty}
\pagestyle{empty}

\begin{abstract}
The design of Nash equilibrium seeking strategies for games in which the involved players are of second-order integrator-type dynamics is investigated in this paper. Noticing that velocity signals are usually noisy or not available for feedback control in practical engineering systems, this paper supposes that the velocity signals are  not accessible for the players. To deal with the absence of velocity measurements, two estimators are designed, based on which Nash equilibrium seeking strategies are constructed. The first strategy is established by employing an observer, which has the same order as the players' dynamics, to estimate the unavailable system states (e.g., the players' velocities). The second strategy is designed based on a high-pass filter and is motivated by the incentive to reduce the order of the closed-loop system which in turn reduces the computation costs of the seeking algorithm. Extensions to Nash equilibrium seeking for networked games are provided. Taking the advantages of leader-following consensus protocols, it turns out that both the observer-based method and the filter-based method
can be adapted to deal with games in distributed systems, which shows the extensibility of the developed strategies. Through Lyapunov stability analysis, it is analytically proven that the players' actions can be regulated to the Nash equilibrium point and their velocities can be regulated to zero by utilizing the proposed velocity-free Nash equilibrium seeking strategies. A numerical example is provided for the verifications of the proposed algorithms.
\end{abstract}


%%%%%%%%%%%%%%%%%%%%%%%%%%%%%%%%%%%%%%%%%%%%%%%%%%%%%%%%%%%%%%%%%%%%%%%%%%%%%%%%
\section{INTRODUCTION}

With the rapid development of Nash equilibrium seeking algorithms in the past few years, games with second-order integrator-type players have drawn some attention recently. In \cite{12}\cite{YEICCA19}, both centralized and distributed Nash equilibrium seeking methods were developed for games with second-order integrator-type dynamics. In particular, a seeking strategy with bounded controls was constructed for the considered game as in practical engineering systems, actuators usually have limited capabilities. In \cite{11}, the authors considered a game in single-input single output dynamical systems with relative degree two. Based on a second-order dynamics with damping coefficients, a control input was designed for the game to achieve centralized Nash equilibrium seeking. It was proven that by utilizing the designed control input with full state feedback, the Nash equilibrium can be stabilized. In \cite{Y}, we considered games in which the players' dynamics appear to be heterogeneous in the sense that some players are of first-order integrator-type dynamics while the rest are second-order integrators. Based on action and velocity feedbacks, Nash equilibrium seeking strategies were proposed for both full information games and partial information games. In \cite{dMB}, games with multiple integrator-type dynamics were concerned and a Nash equilibrium seeking strategy was proposed by employing adaptive control gains.  In \cite{MB}, Nash equilibrium seeking for second-order integrator-type games was addressed by designing methods based on projection operators, consensus protocols as well as primal-dual techniques.
\textbf{However, it is worth mentioning that the above works achieve Nash equilibrium seeking by utilizing full state feedback, i.e., both the players' position information and velocity information should be measured to implement the aforementioned methods, which restricts their applications to some extent as practical situations show that it might be challenging or costly to measure the accurate velocities in real time.}

It is inadvisable to utilize velocity information as in many practical situations, velocity measurements are usually noisy, which may deteriorate the control performance. Moreover, it is costly and complex to install extra velocity sensors in some engineering systems. Actually, quite a few works have been reported to deal with the unavailability of velocity measurements for various control applications. For example, only actuator position measurement units but not velocity measurement devices are included in many commercial robotic systems (e.g., PUMA 560 robot) \cite{Lim97}. To compensate for the limited sensors installed in rigid-link flexible-joint robots, the authors employed a set of filters in the control strategy design to achieve position tracking of the robots \cite{Lim97}. With the development of robots, motion control of mechanical systems without velocity measurement has drawn increasing attention \cite{Andreev17}. Moreover, as angular velocity and relative angular velocities are absent, attitude consensus among a group of spacecraft was addressed by introducing some auxiliary dynamics in \cite{Abdessameud09}. Motivated by the fact that ship velocity measurements are usually unavailable, the authors in  \cite{DoTCST06} designed a controller to drive an underactuated ship along a prescribed path without utilizing ship velocities. Furthermore, as it is challenging to obtain velocity signals for electro-hydraulic servomechanisms, an adaptive strategy was proposed for the tracking control of electro-hydraulic servomechanisms based on extended-state-observers and backstepping techniques in \cite{Deng19}. With the lack of velocity feedback, collaborative control (e.g., consensus, formation, to mention just a few) of second-order multi-agent systems by utilizing only position information was also reported in quite a few works \cite{MeiAT}-\cite{Ren08}.

In spirit of relaxing the requirements on velocity measurements, this paper considers Nash equilibrium seeking for games in which the players are of second-order integrator-type dynamics without utilizing velocity measurements. In comparison with the existing works, the main contributions of the paper are summarized as follows.
\begin{enumerate}
  \item Nash equilibrium seeking for games with second-order integrator-type players is investigated. Compared with the existing works in \cite{12}-\cite{MB}, the velocity measurements are not utilized in the control design, which benefits the applications of games to circumstances in which the players are not equipped with any velocity measurement devices or the measured velocities are noisy. An observer-based approach and a filter-based approach  are proposed to achieve Nash equilibrium seeking based on the estimations of velocities.
  \item Stability of the Nash equilibrium under the proposed seeking strategies is analytically investigated. It is shown through Lyapunov stability analysis that the players' actions can be regulated to the Nash equilibrium and their velocities can be steered to zero by utilizing the proposed methods.
  \item Extensions to partial information games under distributed networks are discussed. By further introducing consensus protocols into the proposed algorithms, we show that both the observer-based approach and the filter-based approach can be adapted to distributed games thus verifying their extensibility. Compared with \cite{5}-\cite{6}, the proposed methods accommodate the players' dynamics without utilizing velocity measurement while in \cite{5}-\cite{6}, the seeking algorithms were designed for games with first-order integrator-type players.
\end{enumerate}

The rest of the paper is organized as follows. The problem is formulated in Section \ref{pr_f} and the main results are given in Section \ref{mar}, in which an observer-based Nash equilibrium seeking strategy and a filter-based approach are proposed for the considered game. In Section \ref{ext}, extensions of the proposed methods to games under distributed networks are provided and in Section \ref{num_ex}, numerical simulations illustrate the effectiveness of the developed algorithms. In the last, Section \ref{conc} provides concluding remarks for the paper.

\section{Problem Formulation}\label{pr_f}

\begin{Problem}
Consider a game with $N$ players in which player $i$'s action is governed by
\begin{equation}\label{d1}
\begin{aligned}
\dot{x}_i=&v_i,\\
\dot{v}_i=&u_i,
\end{aligned}
\end{equation}
for $i\in\mathbb{N}$, where $x_i\in \mathbb{R}$, $v_i\in \mathbb{R}$ and $u_i\in \mathbb{R}$ denote the action, velocity and control input of player $i,$ respectively. Moreover, $\mathbb{N}=\{1,2,\cdots,N\}$ is the set of players involved in the game. Associate player $i$  with a cost function $f_i(\mathbf{x})$, where $i\in\mathbb{N}$ and $\mathbf{x}=[x_1,x_2,\cdots,x_N]^T.$ The objective of this paper is to design Nash equilibrium seeking strategies for the considered game provided that \textbf{the players' velocity measurements are not available.}
\end{Problem}

For notational clarity, let $\mathbf{x}_{-i}=[x_1,x_2,\cdots,x_{i-1},x_{i+1},\cdots,x_{N}]^T$. Then, the Nash equilibrium $\mathbf{x}^*=(x_i^*,\mathbf{x}_{-i}^*)$ is defined as an action profile on which
\begin{equation}
f_i(x_i^*,\mathbf{x}_{-i}^*)\leq f_i(x_i,\mathbf{x}_{-i}^*),
\end{equation}
for $x_i\in \mathbb{R}, i\in\mathbb{N}$. In addition, we say that Nash equilibrium seeking for the considered game is achieved if
\begin{equation}\label{def}
\begin{aligned}
&\lim_{t\rightarrow \infty}||\mathbf{x}(t)-\mathbf{x}^*||=0,\\
&\lim_{t\rightarrow \infty}||\mathbf{v}(t)||=0,
\end{aligned}
\end{equation}
where $\mathbf{v}=[v_1,v_2,\cdots,v_N]^T$.
Furthermore, if the seeking strategy enables \eqref{def} to be satisfied by utilizing only the players' local information, we say that distributed Nash equilibrium seeking is achieved.
\begin{Remark}
Different from \cite{12}-\cite{MB} that utilized full state (including both positions and velocities) feedback in the control law, this paper supposes that the velocity measurements are not available. Note that the concerned problem is of vital importance as practical experiences have shown that velocity measurements tend to contain noises which are difficult to be filtered away. Furthermore, many engineering devices (e.g., robots, ships) are not equipped with velocity measurement units and it might be costly to install additional velocity measurement sensors.
\end{Remark}

For notational convenience, let $\mathcal{P}(\mathbf{x})=\left[\frac{\partial f_1(\mathbf{x})}{\partial x_1},\frac{\partial f_2(\mathbf{x})}{\partial x_2},\cdots,\frac{\partial f_N(\mathbf{x})}{\partial x_N}\right]^T$ and
\begin{equation*}
H(\mathbf{x})=\left[
                       \begin{array}{cccc}
                         \frac{\partial^2 f_{1}(\mathbf{x})}{\partial x_1^2} & \frac{\partial^2 f_{1}(\mathbf{x})}{\partial x_{1}\partial x_2} & \cdots & \frac{\partial^2 f_{1}(\mathbf{x})}{\partial x_{1}\partial x_N} \\
                         \frac{\partial^2 f_{2}(\mathbf{x})}{\partial x_{2}\partial x_1}& \frac{\partial^2 f_{n+2}(\mathbf{x})}{\partial x_2^2} & \cdots & \frac{\partial^2 f_{2}(\mathbf{x})}{\partial x_{2}\partial x_N} \\
                         \vdots &  & \ddots & \vdots \\
                         \frac{\partial^2 f_N(\mathbf{x})}{\partial x_N \partial x_1}  & \frac{\partial^2 f_N(\mathbf{x})}{\partial x_N \partial x_2} & \cdots & \frac{\partial^2 f_N(\mathbf{x})}{\partial x_N^2} \\
                       \end{array}
                     \right].
\end{equation*}

The following provided assumptions will be utilized to develop the main results.

\begin{Assumption}\label{Assu_1}
For each $i\in\mathbb{N},$ $f_i(\mathbf{x})$ is  twice-continuously differentiable.
\end{Assumption}

\begin{Assumption}\label{Assu_2}
There exists a positive constant $m$ such that
\begin{equation}\label{cond_1}
(\mathbf{x}-\mathbf{y})^T(\mathcal{P}(\mathbf{x})-\mathcal{P}(\mathbf{y}))\geq m||\mathbf{x}-\mathbf{y}||^2,
\end{equation}
for $\mathbf{x},\mathbf{y}\in \mathbb{R}^N.$
\end{Assumption}

\begin{Assumption}\label{Assu_3}
There exists a positive constant $h$ such that $\lVert H(\mathbf{x})\lVert$ is upper bounded by $h$, i.e., $\text{sup}_{\mathbf{x}\in \mathbb{R}^N}\lVert H(\mathbf{x})\lVert=h$.
\end{Assumption}

\begin{Remark}
Assumptions \ref{Assu_1}-\ref{Assu_3} are quite mild for games with second-order integrator-type players in the sense that Assumption \ref{Assu_3} can be easily removed by degrading the corresponding results to local/semi-global versions. Note that by Assumption \ref{Assu_3}, we get that for each $i\in\mathbb{N},$ $\frac{\partial f_i(\mathbf{x})}{\partial x_i}$ is globally Lipschitz for $\mathbf{x}\in \mathbb{R}^N$. For notational clarity, we denote the Lipschitz constant of $\frac{\partial f_i(\mathbf{x})}{\partial x_i}$ as $l_i$.
Moreover, Assumption \ref{Assu_2} serves as a commonly utilized condition that results in unique Nash equilibrium on which $\mathcal{P}(\mathbf{x}^*)=\mathbf{0}_N,$ where  $\mathbf{0}_N$ is an $N$-dimensional zero column vector \cite{5}.
\end{Remark}

\section{Main results}\label{mar}
In this section, an observer-based seeking strategy and a filter-based seeking strategy will be successively established to achieve of the goal of the paper.
\subsection{An observer-based Nash equilibrium seeking strategy}\label{mar_obser}
As the players' velocities can not be accessed for feedback in the seeking strategy, it is intuitive that we can design observers to estimate them. Based on this idea, we design the control input of player $i$ for $i\in\mathbb{N}$ as
\begin{equation}\label{d2}
u_i=-k_1\frac{\partial f_i(\mathbf{x})}{\partial x_i}-k_1 \bar{v}_i,
\end{equation}
where $\bar{v}_i$ represents player $i$'s estimate on its own velocity $v_i$ and $k_1$ is a positive constant to be further determined. Moreover, we design the velocity observer as
\begin{equation}\label{d3}
\begin{aligned}
\dot{\bar{x}}_i=&-k_2(\bar{x}_i-x_i)+\bar{v}_i,\\
\dot{\bar{v}}_i=&-k_3(\bar{x}_i-x_i)+u_i,
\end{aligned}
\end{equation}
where $\bar{x}_i$ is an auxiliary variable and $k_2,k_3$ are positive control gains.

In the following, we establish the stability of Nash equilibrium under the proposed method in \eqref{d2}-\eqref{d3}.

\begin{Theorem}\label{th1}
Suppose that Assumptions  \ref{Assu_1}-\ref{Assu_3} are satisfied and
\begin{equation}
k_1>\frac{\epsilon_2(2\epsilon_1h+h\sqrt{N}\max_{i\in\mathbb{N}}\{l_i\}+1)}{\epsilon_1(2\epsilon_2-1)},
\end{equation}
where $\epsilon_1$ and $\epsilon_2$ are arbitrary positive constants that satisfy $\epsilon_1<\frac{2m}{h\sqrt{N}\max_{i\in\mathbb{N}}\{l_i\}+1},$ $\epsilon_2>\frac{1}{2}$. Then, the Nash equilibrium seeking is achieved by \eqref{d2}-\eqref{d3}.
\end{Theorem}
\begin{proof}
Define the observation error as
\begin{equation}
\tilde{x}_i=\bar{x}_i-x_i,\tilde{v}_i=\bar{v}_i-v_i.
\end{equation}

Hence,
\begin{equation}
\begin{aligned}
\dot{\tilde{x}}_i=&\dot{\bar{x}}_i-\dot{x}_i\\
=&-k_2(\bar{x}_i-x_i)+\bar{v}_i-v_i\\
=&-k_2\tilde{x}_i+\tilde{v}_i,
\end{aligned}
\end{equation}
and
\begin{equation}
\begin{aligned}
\dot{\tilde{v}}_i=&\dot{\bar{v}}_i-\dot{v}_i\\
=&-k_3(\bar{x}_i-x_i)=-k_3\tilde{x}_i.
\end{aligned}
\end{equation}

For notational convenience, let $\xi_i=[\tilde{x}_i,\tilde{v}_i]^T$ and define the Lyapunov candidate function as
\begin{equation}
V_1=\sum_{i=1}^N\xi_i^TP\xi_i,
\end{equation}
where $P$ is a symmetric positive definite matrix such that
\begin{equation}
P\left[
   \begin{array}{cc}
     -k_2 & 1 \\
     -k_3 & 0 \\
   \end{array}
 \right]+\left[
   \begin{array}{cc}
     -k_2 & 1 \\
     -k_3 & 0 \\
   \end{array}
 \right]^TP=-Q,
 \end{equation}
and $Q$ is a symmetric positive definite matrix. Note that the existence of $P,Q$ can be concluded by noticing that $\left[
   \begin{array}{cc}
     -k_2 & 1 \\
     -k_3 & 0 \\
   \end{array}
 \right]$ is Hurwitz.
 Then, it can be easily obtained that
 \begin{equation}
 \dot{V}_1=-\sum_{i=1}^N\lambda_{min}(Q)||\xi_i||^2,
 \end{equation}
 from which it is clear that
 \begin{equation}
 \lim_{t\rightarrow \infty} ||\mathbf{\xi}(t)||=0,
 \end{equation}
 where $\mathbf{\xi}=[\xi_1^T,\xi_2^T,\cdots,\xi_N^T]^T.$

 To further proceed the convergence analysis, define
\begin{equation}
\begin{aligned}
V_2=&\frac{1}{2}(\mathbf{v}+\mathcal{P}(\mathbf{x}))^T(\mathbf{v}+\mathcal{P}(\mathbf{x}))\\
&+\frac{1}{2}(\mathbf{x}-\mathbf{x}^*)^T(\mathbf{x}-\mathbf{x}^*).
\end{aligned}
\end{equation}

Let $\bar{\mathbf{v}}=[\bar{v}_1,\bar{v}_2,\cdots,\bar{v}_N]^T$ and $\tilde{\mathbf{v}}=[\tilde{v}_1,\tilde{v}_2,\cdots,\tilde{v}_N]^T$. Then, the time derivative of $V_2$ along the given trajectory is
\begin{equation}
\begin{aligned}
\dot{V}_2=&(\mathbf{v}+\mathcal{P}(\mathbf{x}))^T(-k_1\mathcal{P}(\mathbf{x})-k_1\bar{\mathbf{v}}+H(\mathbf{x})\mathbf{v})\\
&+(\mathbf{x}-\mathbf{x}^*)^T\mathbf{v}\\
\leq &-k_1||\mathbf{v}+\mathcal{P}(\mathbf{x})||^2-m||\mathbf{x}-\mathbf{x}^*||^2\\
&-k_1(\mathbf{v}+\mathcal{P}(\mathbf{x}))^T\tilde{\mathbf{v}}+(\mathbf{v}+\mathcal{P}(\mathbf{x}))^TH(\mathbf{x})\mathbf{v}\\
&+(\mathbf{x}-\mathbf{x}^*)^T(\mathbf{v}+\mathcal{P}(\mathbf{x}))\\
\leq &-(k_1-h)||\mathbf{v}+\mathcal{P}(\mathbf{x})||^2-m||\mathbf{x}-\mathbf{x}^*||^2\\
&+k_1||\mathbf{v}+\mathcal{P}(\mathbf{x})||||\tilde{\mathbf{v}}||\\
&+(h\sqrt{N}\max_{i\in\mathbb{N}}\{l_i\}+1)||\mathbf{v}+\mathcal{P}(\mathbf{x})||||\mathbf{x}-\mathbf{x}^*||,
\end{aligned}
\end{equation}
by utilizing Assumptions \ref{Assu_1}-\ref{Assu_3}.

Noticing that
\begin{equation}
\begin{aligned}
&||\mathbf{v}+\mathcal{P}(\mathbf{x})||||\mathbf{x}-\mathbf{x}^*||\\
\leq & \frac{||\mathbf{v}+\mathcal{P}(\mathbf{x})||^2}{2\epsilon_1}+\frac{\epsilon_1||\mathbf{x}-\mathbf{x}^*||^2}{2},
\end{aligned}
\end{equation}
and
\begin{equation}
\begin{aligned}
&k_1||\mathbf{v}+\mathcal{P}(\mathbf{x})||||\tilde{\mathbf{v}}||\\
\leq & \frac{k_1||\mathbf{v}+\mathcal{P}(\mathbf{x})||^2}{2\epsilon_2}+\frac{k_1\epsilon_2||\tilde{\mathbf{v}}||^2}{2},
\end{aligned}
\end{equation}
where $\epsilon_1,\epsilon_2$ are positive constants that can be arbitrarily chosen, we can get that
\begin{equation}
\begin{aligned}
\dot{V}_2\leq &-\left(k_1-h-\frac{h\sqrt{N}\max_{i\in\mathbb{N}}\{l_i\}+1}{2\epsilon_1}-\frac{k_1}{2\epsilon_2}\right)||\mathbf{v}+\mathcal{P}(\mathbf{x})||^2\\
&-\left(m-\frac{(h\sqrt{N}\max_{i\in\mathbb{N}}\{l_i\}+1)\epsilon_1}{2}\right)||\mathbf{x}-\mathbf{x}^*||^2\\
&+\frac{k_1\epsilon_2}{2}||\tilde{\mathbf{v}}||^2.
\end{aligned}
\end{equation}

Let $\epsilon_1<\frac{2m}{h\sqrt{N}\max_{i\in\mathbb{N}}\{l_i\}+1}$ and $\epsilon_2>\frac{1}{2}.$ Then, for fixed $\epsilon_1,\epsilon_2$, choose $k_1>\frac{\epsilon_2(2\epsilon_1h+h\sqrt{N}\max_{i\in\mathbb{N}}\{l_i\}+1)}{\epsilon_1(2\epsilon_2-1)},$ by which $\rho_1=k_1-h-\frac{h\sqrt{N}\max_{i\in\mathbb{N}}\{l_i\}+1}{2\epsilon_1}-\frac{k_1}{2\epsilon_2}>0$ and $\rho_2=m-\frac{(h\sqrt{N}\max_{i\in\mathbb{N}}\{l_i\}+1)\epsilon_1}{2}>0.$ Hence,
\begin{equation}
\dot{V}_2\leq -\min\{\rho_1,\rho_2\}||E||^2+\frac{k_1\epsilon_2}{2}||\tilde{\mathbf{v}}||^2,
\end{equation}
where $E=[(\mathbf{v}+\mathcal{P}(\mathbf{x}))^T,(\mathbf{x}-\mathbf{x}^*)^T]^T.$

Therefore,
\begin{equation}
\dot{V}_2\leq -\frac{\min\{\rho_1,\rho_2\}}{2}||E||^2,
\end{equation}
for $||E||>\sqrt{\frac{k_1\epsilon_2}{\min\{\rho_1,\rho_2\}}}||\tilde{\mathbf{v}}||$.

Hence, by Theorem 4.19 in \cite{15}, we get that
\begin{equation}
||E(t)||\leq \beta(||E(0)||,t)+\sqrt{\frac{k_1\epsilon_2}{\min\{\rho_1,\rho_2\}}}||\tilde{\mathbf{v}}||,
\end{equation}
where $\beta(\cdot)\in \mathcal{KL}.$

Recalling that
\begin{equation}
\lim_{t\rightarrow \infty} ||\mathbf{\xi}(t)||=0,
\end{equation}
we get that
\begin{equation}
\lim_{t\rightarrow \infty}||E(t)||=0,
\end{equation}
indicating that
\begin{equation}
\lim_{t\rightarrow \infty} ||\mathbf{x}(t)-\mathbf{x}^*||=0,
\end{equation}
and
\begin{equation}
 \lim_{t\rightarrow\infty} ||\mathbf{v}(t)+\mathcal{P}(\mathbf{x})||=0.
\end{equation}

Furthermore, by $\lim_{t\rightarrow \infty} ||\mathbf{x}(t)-\mathbf{x}^*||=0, $ we get that $||\mathcal{P}(\mathbf{x})||\rightarrow 0$ as $t\rightarrow \infty,$ which further indicates that $\lim_{t\rightarrow\infty} ||\mathbf{v}(t)||=0.$ Hence, we arrive at the conclusion.
\end{proof}

In this section, the seeking strategy is designed by constructing a state observer given in \eqref{d3}. It should be noted that the observer is of the same order as the players' dynamics in \eqref{d1}. An intuitive question is whether it is possible to design reduced-order strategies, which would relax the computation costs, to achieve Nash equilibrium seeking or not. In the following section, we provide another strategy design to answer this question.

\subsection{A filter-based Nash equilibrium seeking strategy}

To further reduce the order of the Nash equilibrium seeking strategy, we design the control input of player $i$ for $i\in\mathbb{N}$ as
\begin{equation}\label{eq1}
u_i=-k_1\frac{\partial f_i(\mathbf{x})}{\partial x_i}-k_1y_i,
\end{equation}
where $k_1$ is a positive constant and
\begin{equation}\label{eq2}
y_i=-\hat{x}_i+k_2x_i,
\end{equation}
and $\hat{x}_i$ is an auxiliary variable generated by
\begin{equation}\label{eq3}
\dot{\hat{x}}_i=-k_2\hat{x}_i+k_2^2x_i,
\end{equation}
where $k_2$ is a positive constant to be further determined.

\begin{Remark}
In the control input design \eqref{eq1}, the gradient term is included for the optimization of the players' objective functions. Moreover, $y_i$ serves as an estimate of the velocity of player $i$ and is included to stabilize the system. To provide more insights on how $y_i$ is generated, we can conduct Laplace transformation for \eqref{eq2}-\eqref{eq3}. By \eqref{eq2}, we get that
\begin{equation}\label{eeq1}
Y_i(s)=-\hat{X}_i(s)+k_2X_i(s),
\end{equation}
and by \eqref{eq3}, we get that
\begin{equation}\label{eeq2}
s\hat{X}_i(s)-\hat{x}_i(0)=-k_2\hat{X}_i(s)+k_2^2X_i(s),
\end{equation}
where $s$ is the complex frequency variable and $\hat{X}_i(s)$, $X_i(s)$, $Y_i(s)$ are the signals associated with $\hat{x}_i(t)$, $x_i(t)$, $y_i(t)$ in the complex frequency domain, respectively.
By \eqref{eeq1}-\eqref{eeq2}, it can be easily calculated that
\begin{equation}
Y_i(s)=\frac{sk_2}{s+k_2}X_i(s)-\frac{1}{s+k_2}\hat{x}_i(0),
\end{equation}
where $\frac{s}{s+k_2}$ is a high-pass filter with cut-off frequency  $k_2$. This explains the generation of $y_i(t)$ and why we term the method in \eqref{eq2}-\eqref{eq3} as a filter-based seeking strategy.
\end{Remark}

The following theorem establishes the stability of the Nash equilibrium under the proposed method in \eqref{eq1}-\eqref{eq3}.

\begin{Theorem}\label{the3}
Suppose that Assumptions \ref{Assu_1}-\ref{Assu_3} are satisfied and
\begin{equation}
k_1>h+\frac{\epsilon(h\sqrt{N}\max_{i\in\mathbb{N}}\{l_i\}+1)}{2},k_2>k_1,
\end{equation}
where $\epsilon$ is an arbitrary positive constant that satisfies $\epsilon>\frac{h\sqrt{N}\max_{i\in\mathbb{N}}\{l_i\}+1}{2m}$.
Then, the Nash equilibrium seeking is achieved by \eqref{eq1}-\eqref{eq3}.
\end{Theorem}
\begin{proof}
From \eqref{eq1}-\eqref{eq3}, we can obtain that the concatenated vector form of the closed-loop system can be written as
\begin{equation}\label{eq4}
\begin{aligned}
\dot{\mathbf{x}}=&\mathbf{v},\\
\dot{\mathbf{v}}=&-k_1\mathcal{P}(\mathbf{x})-k_1\mathbf{y}\\
\dot{\mathbf{y}}=&-k_2\mathbf{y}+k_2\mathbf{v},
\end{aligned}
\end{equation}
where $\mathbf{y}=[y_1,y_2,\cdots,y_N]^T.$

Define $\bar{\mathbf{y}}=\mathbf{y}-\mathbf{v}.$ Then, it can be obtained that

\begin{equation}\label{eq5}
\begin{aligned}
\dot{\mathbf{x}}=&\mathbf{v},\\
\dot{\mathbf{v}}=&-k_1\mathcal{P}(\mathbf{x})-k_1\mathbf{v}-k_1\bar{\mathbf{y}}\\
\dot{\bar{\mathbf{y}}}=&-k_2\bar{\mathbf{y}}-(-k_1\mathcal{P}(\mathbf{x})-k_1\mathbf{v}-k_1\bar{\mathbf{y}}).
\end{aligned}
\end{equation}

To establish the stability property for \eqref{eq5}, one can define the Lyapunov candidate function as
\begin{equation}\label{lya_f}
\begin{aligned}
V=&\frac{1}{2}(\mathbf{v}+\mathcal{P}(\mathbf{x}))^T(\mathbf{v}+\mathcal{P}(\mathbf{x}))\\
&+\frac{1}{2}(\mathbf{x}-\mathbf{x}^*)^T(\mathbf{x}-\mathbf{x}^*)+\frac{1}{2}\bar{\mathbf{y}}^T\bar{\mathbf{y}}.
\end{aligned}
\end{equation}

Then, the time derivative of $V$ along the trajectory of \eqref{eq5} is
\begin{equation}
\begin{aligned}
\dot{V}=&(\mathbf{x}-\mathbf{x}^*)^T(\mathbf{v}+\mathcal{P}(\mathbf{x})-\mathcal{P}(\mathbf{x}))\\
&+(\mathbf{v}+\mathcal{P}(\mathbf{x}))^T(-k_1\mathcal{P}(\mathbf{x})-k_1\mathbf{v}+H(\mathbf{x})\mathbf{v})\\
&-(\mathbf{v}+\mathcal{P}(\mathbf{x}))^Tk_1\bar{\mathbf{y}}+\bar{\mathbf{y}}^T(-k_2\bar{\mathbf{y}}-\dot{\mathbf{v}})\\
=&(\mathbf{x}-\mathbf{x}^*)^T(\mathbf{v}+\mathcal{P}(\mathbf{x}))-(\mathbf{x}-\mathbf{x}^*)^T\mathcal{P}(\mathbf{x})\\
&-k_1(\mathbf{v}+\mathcal{P}(\mathbf{x}))^T(\mathbf{v}+\mathcal{P}(\mathbf{x}))+(\mathbf{v}+\mathcal{P}(\mathbf{x}))^TH(\mathbf{x})\mathbf{v}\\
&-(\mathbf{v}+\mathcal{P}(\mathbf{x}))^Tk_1\bar{\mathbf{y}}+\bar{\mathbf{y}}^T(-k_2\bar{\mathbf{y}}-\dot{\mathbf{v}})\\
\leq &-m||\mathbf{x}-\mathbf{x}^*||^2-k_1||\mathbf{v}+\mathcal{P}(\mathbf{x})||^2-(k_2-k_1)||\bar{\mathbf{y}}||^2\\
&+||\mathbf{x}-\mathbf{x}^*||||\mathbf{v}+\mathcal{P}(\mathbf{x})||+(\mathbf{v}+\mathcal{P}(\mathbf{x}))^TH(\mathbf{x})\mathbf{v}\\
\leq &-m||\mathbf{x}-\mathbf{x}^*||^2-k_1||\mathbf{v}+\mathcal{P}(\mathbf{x})||^2\\
&-(k_2-k_1)||\bar{\mathbf{y}}||^2+h||\mathbf{v}+\mathcal{P}(\mathbf{x})||^2\\
&+(h\sqrt{N}\max_{i\in\mathbb{N}}\{l_i\}+1)||\mathbf{v}+\mathcal{P}(\mathbf{x})||||\mathbf{x}-\mathbf{x}^*||,
\end{aligned}
\end{equation}
based on Assumptions \ref{Assu_1}-\ref{Assu_3}.

Noticing that
\begin{equation}
\begin{aligned}
&(h\sqrt{N}\max_{i\in\mathbb{N}}\{l_i\}+1)||\mathbf{v}+\mathcal{P}(\mathbf{x})||||\mathbf{x}-\mathbf{x}^*||\\
\leq & \frac{h\sqrt{N}\max_{i\in\mathbb{N}}\{l_i\}+1}{2\epsilon}||\mathbf{x}-\mathbf{x}^*||^2\\
&+\frac{\epsilon(h\sqrt{N}\max_{i\in\mathbb{N}}\{l_i\}+1)}{2}||\mathbf{v}+\mathcal{P}(\mathbf{x})||^2,
\end{aligned}
\end{equation}
where $\epsilon$ is a positive constant that can be arbitrarily chosen.

Hence,
\begin{equation}
\begin{aligned}
\dot{V}\leq & -\left(m-\frac{h\sqrt{N}\max_{i\in\mathbb{N}}\{l_i\}+1}{2\epsilon}\right)||\mathbf{x}-\mathbf{x}^*||^2\\
&-\left(k_1-h-\frac{\epsilon(h\sqrt{N}\max_{i\in\mathbb{N}}\{l_i\}+1)}{2}\right)||\mathbf{v}+\mathcal{P}(\mathbf{x})||^2\\
&-(k_2-k_1)||\bar{\mathbf{y}}||^2.
\end{aligned}
\end{equation}

Let  $\epsilon>\frac{h\sqrt{N}\max_{i\in\mathbb{N}}\{l_i\}+1}{2m}$ and then for fixed $\epsilon$, choose $k_1>h+\frac{\epsilon(h\sqrt{N}\max_{i\in\mathbb{N}}\{l_i\}+1)}{2}$. Then, for fixed $k_1$, choose $k_2>k_1.$ By the above tuning rule, we get that,
\begin{equation}
\dot{V}\leq -\rho||E||^2,
\end{equation}
where $E=[(\mathbf{v}+\mathcal{P}(\mathbf{x}))^T,(\mathbf{x}-\mathbf{x}^*)^T,\bar{\mathbf{y}}^T]^T$ and $\rho=\min\{m-\frac{h\sqrt{N}\max_{i\in\mathbb{N}}\{l_i\}+1}{2\epsilon},k_1-h-\frac{\epsilon(h\sqrt{N}\max_{i\in\mathbb{N}}\{l_i\}+1)}{2},k_2-k_1\}$.

Hence,
\begin{equation}\label{lya_f1}
||E(t)||\leq e^{-\rho t}||E(0)||,
\end{equation}
by which we can obtain the conclusion.
\end{proof}

\begin{Remark}
From the proof of Theorem \ref{the3}, it can be seen that  $y_i$  in \eqref{eq1}-\eqref{eq3} can be regarded as the estimated value of $v_i.$ Therefore, \eqref{eq3} is designed to drive $y_i$ to $v_i$, which is hard to be accurately measured in practice. By \eqref{lya_f} and \eqref{lya_f1}, we get that $||E(t)||\rightarrow 0$ as $t\rightarrow \infty.$ As $\mathbf{x}(t)\rightarrow \mathbf{x}^*$ for $t\rightarrow \infty,$ we obtain that $||\mathcal{P}(\mathbf{x})||\rightarrow 0$ as $t\rightarrow \infty$ by Assumption \ref{Assu_2}.  Hence, $||\mathbf{v}(t)||\rightarrow 0$ as $t\rightarrow\infty,$ which further indicates that $||\mathbf{y}(t)||\rightarrow 0$ as $t\rightarrow \infty.$
\end{Remark}

\begin{Remark}
Compared with the observer-based approach in \eqref{d2}-\eqref{d3}, we can see that the filter-based approach in \eqref{eq1}-\eqref{eq3} is of less order.  However, it should be mentioned that there are two parameters to be tuned for the filter-based algorithm (see the statement of Theorem \ref{the3}) while the observer-based approach only requires the tuning of one parameter (see the statement in Theorem \ref{th1}).
\end{Remark}
\section{Extensions to games under distributed communication networks}\label{ext}

As the players' objective functions and the gradient values utilized the strategy design depend on all the players' actions, it is necessary to study distributed Nash equilibrium seeking for networked games provided that the players have limited access into the other ones' actions. Hence, in this section, we further consider distributed Nash equilibrium seeking by supposing that player $i$ could not directly get $x_j$ if player $j$ is not its neighbor. Under this setting, $\frac{\partial f_i(\mathbf{x})}{\partial x_i}$ is not available for feedback in the control input design as $\mathbf{x}$ is not available for player $i$. To deal with this situation, we suppose that the players are engaged in a communication network $\mathcal{G},$ defined as a pair $(\mathbb{N},\mathcal{E})$, where $\mathbb{N}$ is the node set and $\mathcal{E}\subseteq \mathbb{N}\times \mathbb{N}$ is the edge set. For an undirected graph, an edge $(i,j)\in \mathcal{E}$ if nodes $i$ and $j$ can receive information from each other. The undirected graph is connected if for any pair of vertices, there exists a path. The adjacency matrix of undirected communication $\mathcal{G}$ is $\mathcal{A}=[a_{ij}],$ where $a_{ij}=1$ if $(j,i)\in \mathcal{E}$, $a_{ij}=0$ if $(j,i)\notin \mathcal{E}$ and $a_{ii}=0$. Moreover, the Laplacian matrix of $\mathcal{G}$ is $\mathcal{L}=\mathcal{D}-\mathcal{A}$, where $\mathcal{D}$ is a diagonal matrix with its $i$th diagonal entry being $d_{ii}=\sum_{j=1}^Na_{ij}.$ In the following, we consider distributed Nash equilibrium seeking strategy design under undirected and connected communication graphs. For notational clarity, define $A_0$ as a diagonal matrix whose diagonal entries are $a_{11},a_{12},\cdots,a_{1N},a_{21},\cdots,a_{NN},$ successively. Moreover, let $I_{N\times N}$ and $\otimes$ be an $N\times N$ dimensional identity matrix and the Kronecker product, respectively. Moreover, for a symmetric real matrix $\Gamma$, $\lambda_{min}(\Gamma)$ defines the minimum eigenvalue of $\Gamma$. In the following, the observer-based method and the filter-based method will be successively adapted for distributed games.

\subsection{An observer-based approach for distributed Nash equilibrium seeking}
Based on the velocity observer design in \eqref{d2}-\eqref{d3} and the distributed seeking strategy in  \cite{5}-\cite{7}, the distributed control input of player $i$ can be designed as
\begin{equation}\label{dd2}
u_i=-k_1\frac{\partial f_i}{\partial x_i}(\mathbf{z}_i)-k_1 \bar{v}_i,
\end{equation}
where $\bar{v}_i$ represents player $i$'s estimate on its own velocity $v_i$, $\frac{\partial f_i}{\partial x_i}(\mathbf{z}_i)=\frac{\partial f_i(\mathbf{x})}{\partial x_i}\left.\right|_{\mathbf{x}=\mathbf{z}_i}$ and $\mathbf{z}_i$ is a vector representing player $i$' estimate on $\mathbf{x}$. Moreover, $\bar{x}_i,\bar{v}_i$ and $\mathbf{z}_i$ are variables generated by
\begin{equation}\label{dd3}
\begin{aligned}
\dot{\bar{x}}_i=&-k_2(\bar{x}_i-x_i)+\bar{v}_i,\\
\dot{\bar{v}}_i=&-k_3(\bar{x}_i-x_i)+u_i,\\
\dot{z}_{ij}=&-k_4(\sum_{k=1}^Na_{ik}(z_{ij}-z_{kj})+a_{ij}(z_{ij}-x_j)),
\end{aligned}
\end{equation}
where $k_4$ is a positive constant, $\mathbf{z}_i=[z_{i1},z_{i2},\cdots,z_{iN}]^T$ and $\bar{x}_i,k_2,k_3$ follows the definitions in Section \ref{mar_obser}.

\begin{Remark}
It is worth mentioning that in \eqref{dd2}-\eqref{dd3}, each player updates its action by utilizing only its local information (e.g., its own information and information from its neighbors). Compared with the strategy in \eqref{d2}-\eqref{d3}, it is clear that the strategy in \eqref{dd2}-\eqref{dd3} serves as the distributed counterpart of \eqref{d2}-\eqref{d3}.
\end{Remark}

Define
\begin{equation}
\tilde{x}_i=\bar{x}_i-x_i,\tilde{v}_i=\bar{v}_i-v_i.
\end{equation}
Then, treating $\tilde{\mathbf{v}},$ defined as $\tilde{\mathbf{v}}=[\tilde{v}_1,\tilde{v}_2,\cdots,\tilde{v}_N]^T$, as an input for the following subsystem
 \begin{equation}\label{sub_s}
 \begin{aligned}
 \dot{x}_i=&v_i,\\
 \dot{v}_i=&-k_1\frac{\partial f_i}{\partial x_i}(\mathbf{z}_i)-k_1 \tilde{v}_i-k_1v_i,\\
 \dot{z}_{ij}=&-k_4(\sum_{k=1}^Na_{ik}(z_{ij}-z_{kj})+a_{ij}(z_{ij}-x_j)),i\in\mathbb{N}
 \end{aligned}
 \end{equation}
it can be shown that \eqref{sub_s} is input-to-state stable by tuning the control gains as illustrated in the following lemma.

 \begin{Lemma}\label{lemma1}
Suppose that Assumptions  \ref{Assu_1}-\ref{Assu_3} are satisfied. Then, \eqref{sub_s} is input-to-state stable by choosing
\begin{equation}
\begin{aligned}
k_1>&\frac{2\epsilon_3}{2\epsilon_3-1}\left(h+\frac{(h\sqrt{N}\max_{i\in\mathbb{N}}\{l_i\}+1)}{2\epsilon_1}\right.\\
&\left.+\frac{\max_{i\in\mathbb{N}}\{l_i\}+\sqrt{N}}{2}\right)\\
k_4>&\frac{\sqrt{N}}{2\lambda_{min}(\mathcal{L}\otimes I_{N\times N}+A_0)}+\frac{N\max_{i\in\mathbb{N}}\{l_i\}}{2\epsilon_2\lambda_{min}(\mathcal{L}\otimes I_{N\times N}+A_0)}\\
&+\frac{k_1^2\max_{i\in\mathbb{N}}\{l_i\}}{2\lambda_{min}(\mathcal{L}\otimes I_{N\times N}+A_0)},
\end{aligned}
\end{equation}
where $\epsilon_1,\epsilon_2,\epsilon_3$ that are positive constants that satisfy
$\frac{(h\sqrt{N}\max_{i\in\mathbb{N}}\{l_i\}+1)\epsilon_1}{2}+\frac{N\max_{i\in\mathbb{N}}\{l_i\}\epsilon_2}{2}<m,$ and $\epsilon_3>\frac{1}{2}.$
\end{Lemma}
\begin{proof}
Define the Lyapunov candidate function as
\begin{equation}
\begin{aligned}
V=&\frac{1}{2}(\mathbf{v}+\mathcal{P}(\mathbf{x}))^T(\mathbf{v}+\mathcal{P}(\mathbf{x}))\\
&+\frac{1}{2}(\mathbf{x}-\mathbf{x}^*)^T(\mathbf{x}-\mathbf{x}^*)\\
&+\frac{1}{2}(\mathbf{z}-\mathbf{1}_N\otimes \mathbf{x})^T(\mathbf{z}-\mathbf{1}_N\otimes \mathbf{x}),
\end{aligned}
\end{equation}
where $\mathbf{z}=[\mathbf{z}_1^T,\mathbf{z}_2^T,\cdots,\mathbf{z}_N^T]^T,$ and $\mathcal{P}(\mathbf{z})=\left[\frac{\partial f_1}{\partial x_1}(\mathbf{z}_1),\frac{\partial f_2}{\partial x_2}(\mathbf{z}_2),\cdots,\frac{\partial f_N}{\partial x_N}(\mathbf{z}_N)\right]^T.$

Then, the time derivative of $V$ along the trajectory of \eqref{sub_s} is
\begin{equation}
\begin{aligned}
V=&(\mathbf{v}+\mathcal{P}(\mathbf{x}))^T(-k_1\mathcal{P}(\mathbf{x})-k_1\bar{\mathbf{v}}+H(\mathbf{x})\mathbf{v})\\
&+(\mathbf{x}-\mathbf{x}^*)^T\mathbf{v}-(\mathbf{z}-\mathbf{1}_N\otimes \mathbf{x})^T\times\\
&(k_4(\mathcal{L}\otimes I_{N\times N}+A_0)(\mathbf{z}-\mathbf{1}_N\otimes \mathbf{x})+\mathbf{1}_N\otimes \mathbf{v})\\
&+k_1(\mathbf{v}+\mathcal{P}(\mathbf{x}))^T (\mathcal{P}(\mathbf{x})-\mathcal{P}(\mathbf{z}))\\
\leq &-k_1||\mathbf{v}+\mathcal{P}(\mathbf{x})||^2-m||\mathbf{x}-\mathbf{x}^*||^2\\
&-k_1(\mathbf{v}+\mathcal{P}(\mathbf{x}))^T\tilde{\mathbf{v}}+(\mathbf{v}+\mathcal{P}(\mathbf{x}))^TH(\mathbf{x})\mathbf{v}\\
&+(\mathbf{x}-\mathbf{x}^*)^T(\mathbf{v}+\mathcal{P}(\mathbf{x}))\\
&-k_4\lambda_{min}(\mathcal{L}\otimes I_{N\times N}+A_0)||\mathbf{z}-\mathbf{1}_N\otimes \mathbf{x}||^2\\
&-(\mathbf{z}-\mathbf{1}_N\otimes \mathbf{x})^T\mathbf{1}_N\otimes \mathbf{v}\\
&+k_1(\mathbf{v}+\mathcal{P}(\mathbf{x}))^T (\mathcal{P}(\mathbf{x})-\mathcal{P}(\mathbf{z}))\\
\leq &-(k_1-h)||\mathbf{v}+\mathcal{P}(\mathbf{x})||^2-m||\mathbf{x}-\mathbf{x}^*||^2\\
&+k_1||\mathbf{v}+\mathcal{P}(\mathbf{x})||||\tilde{\mathbf{v}}||\\
&+(h\sqrt{N}\max_{i\in\mathbb{N}}\{l_i\}+1)||\mathbf{v}+\mathcal{P}(\mathbf{x})||||\mathbf{x}-\mathbf{x}^*||\\
&-k_4\lambda_{min}(\mathcal{L}\otimes I_{N\times N}+A_0)||\mathbf{z}-\mathbf{1}_N\otimes \mathbf{x}||^2\\
&+\sqrt{N}||\mathbf{z}-\mathbf{1}_N\otimes \mathbf{x}|| ||\mathcal{P}(\mathbf{x})+\mathbf{v}||\\
&+N\max_{i\in\mathbb{N}}\{l_i\}||\mathbf{z}-\mathbf{1}_N\otimes \mathbf{x}||||\mathbf{x}-\mathbf{x}^*||\\
&+k_1\max_{i\in\mathbb{N}}\{l_i\}||\mathbf{v}+\mathcal{P}(\mathbf{x})||||\mathbf{z}-\mathbf{1}_N\otimes \mathbf{x}||,
\end{aligned}
\end{equation}
where $\lambda_{min}(\mathcal{L}\otimes I_{N\times N}+A_0)>0$  as the communication graph $\mathcal{G}$ is undirected and connected.

Therefore,
\begin{equation}
\begin{aligned}
\dot{V}_2\leq &-\rho_1||\mathbf{v}+\mathcal{P}(\mathbf{x})||^2-\rho_2||\mathbf{x}-\mathbf{x}^*||^2\\
&-\rho_3||\mathbf{z}-\mathbf{1}_N\otimes \mathbf{x}||^2+\frac{k_1\epsilon_3}{2}||\tilde{\mathbf{v}}||^2,
\end{aligned}
\end{equation}
where $\rho_1=k_1-h-\frac{(h\sqrt{N}\max_{i\in\mathbb{N}}\{l_i\}+1)}{2\epsilon_1}-\frac{\sqrt{N}}{2}-\frac{\max_{i\in\mathbb{N}}\{l_i\}}{2}-\frac{k_1}{2\epsilon_3},$ $\rho_2=m-\frac{(h\sqrt{N}\max_{i\in\mathbb{N}}\{l_i\}+1)\epsilon_1}{2}-\frac{N\max_{i\in\mathbb{N}}\{l_i\}\epsilon_2}{2},$ $\rho_3=k_4\lambda_{min}(\mathcal{L}\otimes I_{N\times N}+A_0)-\frac{\sqrt{N}}{2}-\frac{N\max_{i\in\mathbb{N}}\{l_i\}}{2\epsilon_2}-\frac{k_1^2\max_{i\in\mathbb{N}}\{l_i\}}{2}$ and $\epsilon_1,\epsilon_2,\epsilon_3$ are positive constants that can be arbitrarily chosen. Choose $\epsilon_1,\epsilon_2$ to be sufficiently small such that $\rho_2>0$ and choose $\epsilon_3>\frac{1}{2}.$ Then, for fixed $\epsilon_1,\epsilon_3,$ choose
$k_1>\frac{2\epsilon_3}{2\epsilon_3-1}\left(h+\frac{(h\sqrt{N}\max_{i\in\mathbb{N}}\{l_i\}+1)}{2\epsilon_1}+\frac{\sqrt{N}}{2}+\frac{\max_{i\in\mathbb{N}}\{l_i\}}{2}\right).$ Then, for fixed $k_1$, choose $k_4>\frac{\sqrt{N}}{2\lambda_{min}(\mathcal{L}\otimes I_{N\times N}+A_0)}+\frac{N\max_{i\in\mathbb{N}}\{l_i\}}{2\epsilon_2\lambda_{min}(\mathcal{L}\otimes I_{N\times N}+A_0)}+\frac{k_1^2\max_{i\in\mathbb{N}}\{l_i\}}{2\lambda_{min}(\mathcal{L}\otimes I_{N\times N}+A_0)}.$ By the above tuning rule, we get that $\rho_1>0,\rho_2>0,\rho_3>0.$

Hence,
\begin{equation}
\dot{V}_2\leq -\frac{\min\{\rho_1,\rho_2,\rho_3\}}{2}||E_1||^2,
\end{equation}
for $||E_1||\geq \sqrt{\frac{k_1\epsilon_3}{\min\{\rho_1,\rho_2,\rho_3\}}}||\tilde{\mathbf{v}}||$, where $E_1=[(\mathbf{v}+\mathcal{P}(\mathbf{x}))^T, (\mathbf{x}-\mathbf{x}^*)^T,(\mathbf{z}-\mathbf{1}_N\otimes \mathbf{x})^T]^T$. Hence, by Theorem 4.19 in \cite{15}, there exists a $\mathcal{KL}$ function $\beta_1$ such that
\begin{equation}
||E_1(t)||\leq \beta_1(||E_1(0)||,t)+\sqrt{\frac{k_1\epsilon_3}{\min\{\rho_1,\rho_2,\rho_3\}}}||\tilde{\mathbf{v}}(t)||,
\end{equation}
thus arriving at the conclusion.
\end{proof}

We are now ready to provide the stability property of Nash equilibrium under the proposed method in \eqref{dd2}-\eqref{dd3}.

\begin{Theorem}\label{th4}
Suppose that Assumptions  \ref{Assu_1}-\ref{Assu_3} are satisfied. Then, the distributed Nash equilibrium seeking is achieved by utilizing \eqref{dd2}-\eqref{dd3} given that
\begin{equation}
\begin{aligned}
k_1>&\frac{2\epsilon_3}{2\epsilon_3-1}\left(h+\frac{(h\sqrt{N}\max_{i\in\mathbb{N}}\{l_i\}+1)}{2\epsilon_1}\right.\\
&\left.+\frac{\sqrt{N}+\max_{i\in\mathbb{N}}\{l_i\}}{2}\right)\\
k_4>&\frac{\sqrt{N}}{2\lambda_{min}(\mathcal{L}\otimes I_{N\times N}+A_0)}+\frac{N\max_{i\in\mathbb{N}}\{l_i\}}{2\epsilon_2\lambda_{min}(\mathcal{L}\otimes I_{N\times N}+A_0)}\\
&+\frac{k_1^2\max_{i\in\mathbb{N}}\{l_i\}}{2\lambda_{min}(\mathcal{L}\otimes I_{N\times N}+A_0)},
\end{aligned}
\end{equation}
where $\epsilon_1,\epsilon_2,\epsilon_3$ are positive constants that satisfy
$\frac{(h\sqrt{N}\max_{i\in\mathbb{N}}\{l_i\}+1)\epsilon_1}{2}+\frac{N\max_{i\in\mathbb{N}}\{l_i\}\epsilon_2}{2}<m,$ and $\epsilon_3>\frac{1}{2}.$
\end{Theorem}
\begin{proof}
Based on Lemma \ref{lemma1}, the proof can be completed by following similar arguments as those in the proof of Theorem \ref{th1}.
\end{proof}

This section provides a distributed counterpart for the observer based approach. In the following, we adapt the filter-based approach for distributed games.

\subsection{A filter-based distributed Nash equilibrium seeking strategy}
Motivated by the filter-based strategy in \eqref{eq1}-\eqref{eq3}  and the distributed seeking strategy in  \cite{5}-\cite{7}, the control input of player $i$ for $i\in\mathbb{N}$ can be designed as
\begin{equation}\label{Y1}
u_i=-k_1\frac{\partial f_i}{\partial x_i}(\mathbf{z}_i)-k_1y_i,
\end{equation}
where $k_1$ is a positive constant, $\mathbf{z}_i=[z_{i1},z_{i2},\cdots,z_{iN}]^T$ and
\begin{equation}\label{Y2}
\begin{aligned}
y_i=&-\hat{x}_i+k_2x_i\\
\dot{z}_{ij}=&-k_3 \left(\sum_{k=1}^Na_{ik}(z_{ij}-z_{kj})+a_{ij}(z_{ij}-x_j)\right),
\end{aligned}
\end{equation}
where $k_2,k_3$ are positive constants and $\hat{x}_i$ is an auxiliary variable generated by
\begin{equation}\label{Y3}
\dot{\hat{x}}_i=-k_2\hat{x}_i+k_2^2x_i.
\end{equation}

The following theorem establishes the stability of the Nash equilibrium under the proposed method in \eqref{Y1}-\eqref{Y3}.
\begin{Theorem}\label{YTH4}
Suppose that Assumptions \ref{Assu_1}-\ref{Assu_3} are satisfied.  Then, the distributed Nash equilibrium seeking is achieved
% and
%\begin{equation}
%\begin{aligned}
%&\lim_{t\rightarrow\infty}||\mathbf{y}(t)||=0,\\
%&\lim_{t\rightarrow\infty}||\mathbf{z}(t)-\mathbf{1}_N\otimes \mathbf{x}(t)||=0,
%\end{aligned}
%\end{equation}
by utilizing the proposed method in \eqref{Y1}-\eqref{Y3} provided that the control gains are designed according to
\begin{equation}
\begin{aligned}
k_1>&\frac{(h\sqrt{N}\max_{i\in\mathbb{N}}\{l_i\}+1)^2}{4m}+h,k_2>k_1,\\
k_3>&\frac{(2k_1\max_{i\in\mathbb{N}}\{l_i\}+\sqrt{N}+N\max_{i\in\mathbb{N}} \{l_i\})^2}{4\min\{\lambda_{min}(A),k_2-k_1\}\lambda_{min}(\mathcal{L}\otimes I_{N\times N}+A_0)},
\end{aligned}
\end{equation}
where $A=\left[
            \begin{array}{cc}
              m & -\frac{h\sqrt{N}\max_{i\in\mathbb{N}}\{l_i\}+1}{2} \\
              -\frac{h\sqrt{N}\max_{i\in\mathbb{N}}\{l_i\}+1}{2} & k_1-h \\
            \end{array}
          \right].
$
\end{Theorem}
\begin{proof}
From \eqref{Y1}-\eqref{Y3}, we can obtain that the concatenated vector form of the closed-loop system can be written as
\begin{equation}\label{Y4}
\begin{aligned}
\dot{\mathbf{x}}=&\mathbf{v}\\
\dot{\mathbf{v}}=&-k_1\mathcal{P}(\mathbf{z})-k_1\mathbf{y}\\
\dot{\mathbf{y}}=&-k_2\mathbf{y}+k_2\mathbf{v}\\
\mathbf{\dot{z}}=&-k_3(\mathcal{L}\otimes I_{N\times N}+A_0)(\mathbf{z}-\mathbf{1}_N\otimes \mathbf{x}),
\end{aligned}
\end{equation}
where $\mathbf{y}=[y_1,y_2,\cdots,y_N]^T$, and $\mathcal{P}(\mathbf{z})=\left[\frac{\partial f_1}{\partial x_1}(\mathbf{z}_1),\frac{\partial f_2}{\partial x_2}(\mathbf{z}_2),\cdots,\frac{\partial f_N}{\partial x_N}(\mathbf{z}_N)\right]^T.$

Define $\bar{\mathbf{y}}=\mathbf{y}-\mathbf{v},$ then, it can be obtained that

\begin{equation}\label{Y5}
\begin{aligned}
\dot{\mathbf{x}}=&\mathbf{v}\\
\dot{\mathbf{v}}=&-k_1\mathcal{P}(\mathbf{z})-k_1\mathbf{v}-k_1\bar{\mathbf{y}}\\
\dot{\bar{\mathbf{y}}}=&-k_2\bar{\mathbf{y}}-(-k_1\mathcal{P}(\mathbf{z})-k_1\mathbf{v}-k_1\bar{\mathbf{y}})\\
\mathbf{\dot{z}}=&-k_3(\mathcal{L}\otimes I_{N\times N}+A_0)(\mathbf{z}-\mathbf{1}_N\otimes \mathbf{x}).
\end{aligned}
\end{equation}

To establish the stability property for \eqref{Y5}, one can define the Lyapunov candidate function as
\begin{equation}
\begin{aligned}
V=V_1+V_2+V_3+V_4,
\end{aligned}
\end{equation}
in which
\begin{equation}
\begin{aligned}
V_1=&\frac{1}{2}(\mathbf{v}+\mathcal{P}(\mathbf{x}))^T(\mathbf{v}+\mathcal{P}(\mathbf{x}))\\
V_2=&\frac{1}{2}(\mathbf{x}-\mathbf{x}^*)^T(\mathbf{x}-\mathbf{x}^*),V_3=\frac{1}{2}\bar{\mathbf{y}}^T\bar{\mathbf{y}}\\
V_4=&\frac{1}{2}(\mathbf{z}-\mathbf{1}_N\otimes \mathbf{x})^T (\mathbf{z}-\mathbf{1}_N\otimes \mathbf{x}).
\end{aligned}
\end{equation}

Then, following the analysis in the proof of Lemma \ref{lemma1} and Theorem \ref{the3}, we get that
\begin{equation}
\begin{aligned}
\dot{V}_1
%&(\mathbf{v}+\mathcal{P}(\mathbf{x}))^T(\dot{\mathbf{v}}+H(\mathbf{x})\mathbf{v})\\
%=&-k_1(\mathbf{v}+\mathcal{P}(\mathbf{x}))^T(\mathcal{P}(\mathbf{x})+\mathbf{v}+\mathcal{P}(\mathbf{z})-\mathcal{P}(\mathbf{x}))\\
%&+(\mathbf{v}+\mathcal{P}(\mathbf{x}))^TH(\mathbf{x})(\mathbf{v}+\mathcal{P}(\mathbf{x})+\mathcal{P}(\mathbf{x}^*)-\mathcal{P}(\mathbf{x}))\\
%&-k_1(\mathbf{v}+\mathcal{P}(\mathbf{x}))^T\bar{\mathbf{y}}\\
\leq&-k_1\Vert \mathbf{v}+\mathcal{P}(\mathbf{x})\Vert^2-k_1(\mathbf{v}+\mathcal{P}(\mathbf{x}))^T\bar{\mathbf{y}}\\
&+k_1\max_{i\in\mathbb{N}}\{l_i\}\Vert \mathbf{v}+\mathcal{P}(\mathbf{x})\Vert\Vert \mathbf{z}-\mathbf{1}_N\otimes \mathbf{x}\Vert\\
&+h\Vert \mathbf{v}+\mathcal{P}(\mathbf{x})\Vert^2\\
&+h\sqrt{N}\max_{i\in\mathbb{N}}\{l_i\}\Vert \mathbf{v}+\mathcal{P}(\mathbf{x})\Vert\Vert \mathbf{x}-\mathbf{x}^*\Vert,\\
\end{aligned}
\end{equation}
and
\begin{equation}
\begin{aligned}
\dot{V}_2
%=&(\mathbf{x}-\mathbf{x}^*)^T(\mathbf{v}+\mathcal{P}(\mathbf{x}))-(\mathbf{x}-\mathbf{x}^*)^T\mathcal{P}(\mathbf{x})\\
\leq &-m||\mathbf{x}-\mathbf{x}^*||^2+||\mathbf{x}-\mathbf{x}^*||||\mathbf{v}+\mathcal{P}(\mathbf{x})||.
\end{aligned}
\end{equation}
Moreover,
\begin{equation}
\begin{aligned}
\dot{V}_3
%=&\bar{\mathbf{y}}^T(-k_2\bar{\mathbf{y}}+k_1\mathcal{P}(\mathbf{z})+k_1\mathbf{v}+k_1\bar{\mathbf{y}})\\
%=&-(k_2-k_1)\Vert \bar{\mathbf{y}}\Vert^2\\
%&+k_1\bar{\mathbf{y}}^T(\mathbf{v}+\mathcal{P}(\mathbf{x})+\mathcal{P}(\mathbf{z})-\mathcal{P}(\mathbf{x}))\\
\leq&-(k_2-k_1)\Vert \bar{\mathbf{y}}\Vert^2+k_1\bar{\mathbf{y}}^T(\mathbf{v}+\mathcal{P}(\mathbf{x}))\\
&+k_1\max_{i\in\mathbb{N}}\{l_i\}\Vert \bar{\mathbf{y}}\Vert\Vert\mathbf{z}-\mathbf{1}_N\otimes \mathbf{x}\Vert.
\end{aligned}
\end{equation}
Furthermore,
\begin{equation}
\begin{aligned}
\dot{V}_4\leq&-k_3\lambda_{min}(\mathcal{L}\otimes I_{N\times N}+A_0){\lVert \mathbf{z}-\mathbf{1}_N \otimes \mathbf{x} \rVert}^2\\
&+\sqrt{N}\Vert\mathbf{z}-\mathbf{1}_N\otimes \mathbf{x}\Vert\Vert \mathbf{v}+\mathcal{P}(\mathbf{x})\Vert\\
&+N\max_{i\in\mathbb{N}}\{l_i\}\Vert\Vert\mathbf{z}-\mathbf{1}_N\otimes \mathbf{x}\Vert\Vert \mathbf{x}-\mathbf{x}^*\Vert.
\end{aligned}
\end{equation}

Hence,
\begin{equation}
\begin{aligned}
\dot{V}\leq&-(k_1-h)\Vert \mathbf{v}+\mathcal{P}(\mathbf{x})\Vert^2-m||\mathbf{x}-\mathbf{x}^*||^2-(k_2-k_1)\Vert \bar{\mathbf{y}}\Vert^2\\
&-k_3\lambda_{min}(\mathcal{L}\otimes I_{N\times N}+A_0){\lVert \mathbf{z}-\mathbf{1}_N \otimes \mathbf{x} \rVert}^2\\
&+(k_1\max_{i\in\mathbb{N}}\{l_i\}+\sqrt{N})\Vert \mathbf{v}+\mathcal{P}(\mathbf{x})\Vert\Vert \mathbf{z}-\mathbf{1}_N\otimes \mathbf{x}\Vert\\
&+(h\sqrt{N}\max_{i\in\mathbb{N}}\{l_i\}+1)\Vert \mathbf{v}+\mathcal{P}(\mathbf{x})\Vert\Vert \mathbf{x}-\mathbf{x}^*\Vert\\
&+k_1\max_{i\in\mathbb{N}}\{l_i\}\Vert \bar{\mathbf{y}}\Vert\Vert\mathbf{z}-\mathbf{1}_N\otimes \mathbf{x}\Vert\\
&+N\max_{i\in\mathbb{N}} \{l_i\}\Vert\mathbf{z}-\mathbf{1}_N\otimes \mathbf{x}\Vert\Vert \mathbf{x}-\mathbf{x}^*\Vert.
\end{aligned}
\end{equation}
Define $A=\left[
            \begin{array}{cc}
              m & -\frac{h\sqrt{N}\max_{i\in\mathbb{N}}\{l_i\}+1}{2} \\
              -\frac{h\sqrt{N}\max_{i\in\mathbb{N}}\{l_i\}+1}{2} & k_1-h \\
            \end{array}
          \right].
$
Then, $A$ is symmetric positive definite by choosing $k_1>\frac{(h\sqrt{N}\max_{i\in\mathbb{N}}\{l_i\}+1)^2}{4m}+h.$ If this is the case,
\begin{equation}
\begin{aligned}
V\leq &-\lambda_{min}(A)||E_1||^2-(k_2-k_1)\Vert \bar{\mathbf{y}}\Vert^2\\
&-k_3\lambda_{min}(\mathcal{L}\otimes I_{N\times N}+A_0){\lVert \mathbf{z}-\mathbf{1}_N \otimes \mathbf{x} \rVert}^2\\
&+(k_1\max_{i\in\mathbb{N}}\{l_i\}+\sqrt{N})\Vert \mathbf{v}+\mathcal{P}(\mathbf{x})\Vert\Vert \mathbf{z}-\mathbf{1}_N\otimes \mathbf{x}\Vert\\
&+k_1\max_{i\in\mathbb{N}}\{l_i\}\Vert \bar{\mathbf{y}}\Vert\Vert\mathbf{z}-\mathbf{1}_N\otimes \mathbf{x}\Vert\\
&+N\max_{i\in\mathbb{N}} \{l_i\}\Vert\mathbf{z}-\mathbf{1}_N\otimes \mathbf{x}\Vert\Vert \mathbf{x}-\mathbf{x}^*\Vert.
\end{aligned}
\end{equation}
where $E_1=[(\mathbf{x}-\mathbf{x}^*)^T,(\mathbf{v}+\mathcal{P}(\mathbf{x}))^T]^T$.

Choose $k_2>k_1$, then,
\begin{equation}
\begin{aligned}
V\leq &-\min\{\lambda_{min}(A),k_2-k_1\}||E_2||^2\\
&-k_3\lambda_{min}(\mathcal{L}\otimes I_{N\times N}+A_0){\lVert \mathbf{z}-\mathbf{1}_N \otimes \mathbf{x} \rVert}^2\\
&+(k_1\max_{i\in\mathbb{N}}\{l_i\}+\sqrt{N})\Vert \mathbf{v}+\mathcal{P}(\mathbf{x})\Vert\Vert \mathbf{z}-\mathbf{1}_N\otimes \mathbf{x}\Vert\\
&+k_1\max_{i\in\mathbb{N}}\{l_i\}\Vert \bar{\mathbf{y}}\Vert\Vert\mathbf{z}-\mathbf{1}_N\otimes \mathbf{x}\Vert\\
&+N\max_{i\in\mathbb{N}} \{l_i\}\Vert\mathbf{z}-\mathbf{1}_N\otimes \mathbf{x}\Vert\Vert \mathbf{x}-\mathbf{x}^*\Vert.
\end{aligned}
\end{equation}
where $E_2=[E_1^T,\bar{\mathbf{y}}^T]^T$.

Hence, by choosing
\begin{equation}
k_3>\frac{(2k_1\max_{i\in\mathbb{N}}\{l_i\}+\sqrt{N}+N\max_{i\in\mathbb{N}} \{l_i\})^2}{4\min\{\lambda_{min}(A),k_2-k_1\}\lambda_{min}(\mathcal{L}\otimes I_{N\times N}+A_0)},
\end{equation}
we get that
\begin{equation}
\begin{aligned}
V\leq &-\lambda_{min}(A_1)||E_3||^2\\
\end{aligned}
\end{equation}
where $\lambda_{min}(A_1)>0,$  $E_3=[E_2^T,\mathbf{z}-(\mathbf{1}_N\otimes \mathbf{x})^T]^T$ and
$A_1=\left[
            \begin{array}{cc}
              \min\{\lambda_{min}(A),k_2-k_1\} & \chi \\
              \chi & k_3\lambda_{min}(\mathcal{L}\otimes I_{N\times N}+A_0) \\
            \end{array}
          \right],
$
where $\chi=-\frac{2k_1\max_{i\in\mathbb{N}}\{l_i\}+\sqrt{N}+N\max_{i\in\mathbb{N}} \{l_i\}}{2}.$

Recalling the definition of the Lyapunov candidate function, the conclusion can be obtained.
\end{proof}
\begin{Remark}
Distributed Nash equilibrium seeking in this paper is achieved based on the idea from \cite{5}-\cite{7} to distributively obtain position estimates via leader-following consensus algorithms. It is worth mentioning that in \cite{5}-\cite{7}, the players are considered as first-order integrators and hence, the Nash equilibrium seeking strategy can be freely designed. Different from \cite{5}-\cite{7}, this paper considers that the players are second-order integrators. With the players' inherent dynamics involved, the Nash equilibrium seeking algorithm should not only drive the players' positions to the Nash equilibrium but also steer their velocities to zero. This indicates that stabilization of the players' dynamics and optimization of the players' cost functions should be achieved simultaneously. In particular, the stabilization of the players' dynamics usually requires the feedback of the players' velocities, which are difficult to be accurately measured in practice. Hence, this paper designs the distributed algorithms without utilizing velocity measurement, which makes the problem more complex. Note that the communication graph is supposed to be fixed in this paper and switching communication topologies (see e.g., \cite{6}\cite{WANGSMCA}-\cite{WenSMACA}) will be addressed in future works.
\end{Remark}
\section{A Numerical example}\label{num_ex}
In this section, the connectivity control game among networked acceleration-actuated mobile sensors considered in \cite{12}-\cite{YEICCA19}\cite{7} is simulated. More specifically, we consider a game with five players whose cost functions are given as
\begin{equation}
\begin{aligned}
&f_{1}(x)=x_{11}^2+x_{11}+2x_{12}^2+x_{12}+1+\Vert x_{1}-x_{3}\Vert^2\\
&f_{2}(x)=3x_{21}^2+2x_{21}+3x_{22}^2+3x_{22}+2+\Vert x_{2}-x_{3}\Vert^2\\
&f_{3}(x)=5x_{31}^2+2x_{31}+5x_{32}^2+2x_{32}+3+\Vert x_{3}-x_{1}\Vert^2\\
&f_{4}(x)=6x_{41}^2+4x_{41}+6x_{42}^2+4x_{42}+4+\Vert x_{4}-x_{2}\Vert^2\\
&f_{5}(x)=8x_{51}^2+6x_{51}+8x_{52}^2+6x_{52}+5+\Vert x_{5}-x_{4}\Vert^2,
\end{aligned}
\end{equation}
respectively.

The unique Nash equilibrium of the game is $[-0.363,-0.235,-0.307,-0.426,-0.227,-0.206,-0.329,$ $-0.347,-0.370,-0.372]^T.$ In the following, we will simulate the centralized algorithms and their distributed counterparts, successively.
\subsection{Centralized Nash equilibrium seeking}
\subsubsection{An observer-based Nash equilibrium seeking strategy}

This section provides numerical verifications for the observer-based method in \eqref{d2}-\eqref{d3}.
In the numerical study, we let $\mathbf{x}(0)=[-0.5,0.5,-1,0,1,0,0,-1,-1,-1.5]^T$. Moreover, other variables in \eqref{d2}-\eqref{d3} are initialized to be zero. The simulation results are given in Figs. \ref{obser_t}-\ref{obser_v}, which plot the players' positions and their velocities, respectively. From Figs. \ref{obser_t}-\ref{obser_v}, we can see that the Nash equilibrium seeking is achieved by the observer-based method in \eqref{d2}-\eqref{d3}.

\begin{figure}
  \centering
  \includegraphics[scale=0.52]{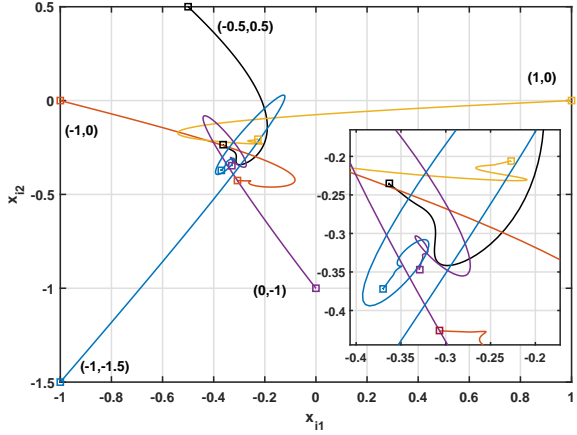}\\
  \caption{The trajectories of players' positions generated by \eqref{d2}-\eqref{d3}. }\label{obser_t}
\end{figure}

\begin{figure}
  \centering
  \includegraphics[scale=0.5]{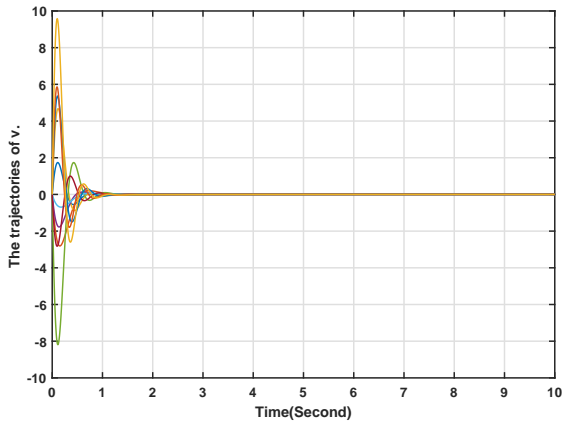}\\
  \caption{The players' velocities generated by \eqref{d2}-\eqref{d3}. }\label{obser_v}
\end{figure}

\subsubsection{A filter-based Nash equilibrium seeking strategy}
This section illustrates the effectiveness of the filter-based method in \eqref{eq1}-\eqref{eq3}. In the simulation, we let $\mathbf{x}(0)=[-0.5,0.5,-1,0,1,0,0,-1,-1,-1.5]^T$. Furthermore, all the other variables in \eqref{eq1}-\eqref{eq3} are initialized at zero. The simulation results generated by \eqref{eq1}-\eqref{eq3} are shown in Figs. \ref{est_1}-\ref{est_2}, which plot the players' positions and velocities, respectively. From the simulation results, we see that the Nash equilibrium seeking can be achieved by utilizing the filter-based method in \eqref{eq1}-\eqref{eq3}.

\begin{figure}
  \centering
  \includegraphics[scale=0.5]{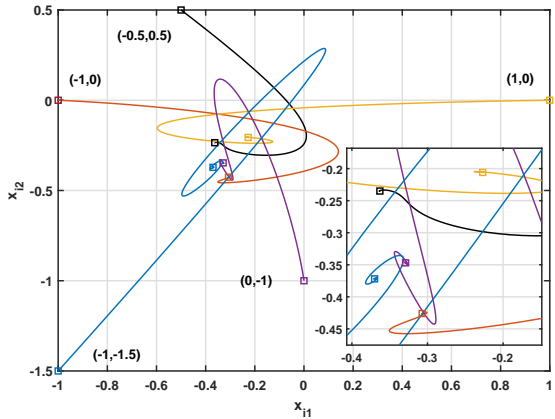}\\
  \caption{The trajectories of players' positions generated by \eqref{eq1}-\eqref{eq3}. }\label{est_1}
\end{figure}

\begin{figure}
  \centering
  \includegraphics[scale=0.5]{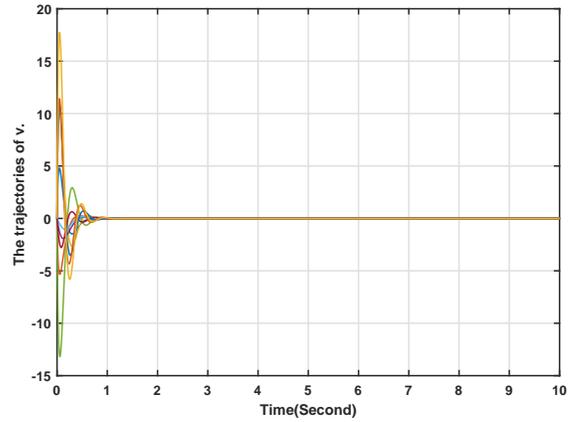}\\
  \caption{The players' velocities generated by \eqref{eq1}-\eqref{eq3}. }\label{est_2}
\end{figure}

\subsection{Distributed Nash equilibrium seeking}
In this section, we provide simulation results for the distributed Nash equilibrium seeking strategies. In the simulations
, the communication topology is depicted in Fig. \ref{G}.
\begin{figure}
  \centering
  \includegraphics[scale=0.35]{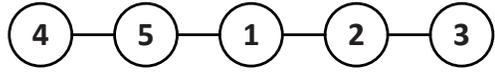}\\
  \caption{The communication graph for the players. }\label{G}
\end{figure}
\subsubsection{An observer-based approach for distributed Nash equilibrium seeking}
This section provides simulation results for the method in  \eqref{dd2}-\eqref{dd3}. In the simulation,  $\mathbf{x}(0)=[-0.5,0.5,-1,0,1,0,0,-1,-1,-1.5]^T,$ and other variables are initialized at zero.
Generated by \eqref{dd2}-\eqref{dd3}, the simulation results are given in Figs. \ref{14_tra}-\ref{14_tra1}, which plot the players' positions and velocities, respectively. From the figures, it can be seen that the Nash equilibrium seeking can be achieved by utilizing the observer-based method in \eqref{dd2}-\eqref{dd3} in a distributed fashion.

\begin{figure}
  \centering
  \includegraphics[scale=0.52]{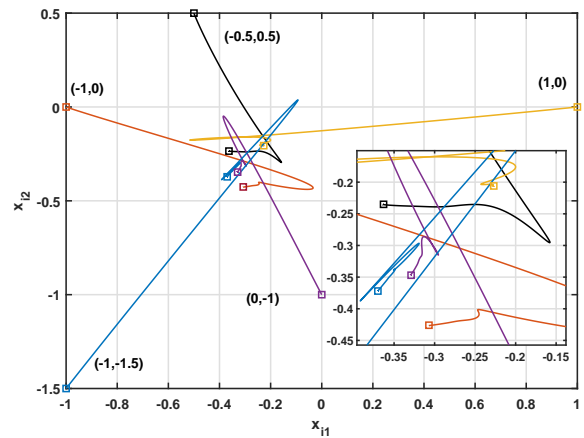}\\
  \caption{The trajectories of players' positions generated by \eqref{dd2}-\eqref{dd3}. }\label{14_tra}
\end{figure}
\begin{figure}
  \centering
  \includegraphics[scale=0.5]{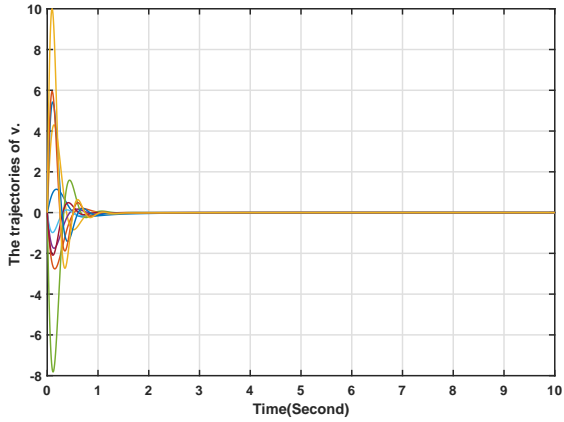}\\
  \caption{The players' velocities generated by \eqref{dd2}-\eqref{dd3}. }\label{14_tra1}
\end{figure}

\subsubsection{A filter-based approach for distributed Nash equilibrium seeking}

This section provides numerical verification for the distributed method in \eqref{Y1}-\eqref{Y2}. In the numerical study,  $\mathbf{x}(0)=[-0.5,0.5,-1,0,1,0,0,-1,-1,-1.5]^T$ and the initial values of other variables in  \eqref{Y1}-\eqref{Y2} are set to be zero. The simulation results generated by \eqref{Y1}-\eqref{Y2} are shown in Figs. \ref{12_tra}-\ref{12_tra1}, which illustrate the players' positions and velocities, respectively. From the figures, it is clear that the Nash equilibrium seeking is achieved in a distributed fashion by utilizing the method in \eqref{Y1}-\eqref{Y2}.

\begin{figure}
  \centering
  \includegraphics[scale=0.5]{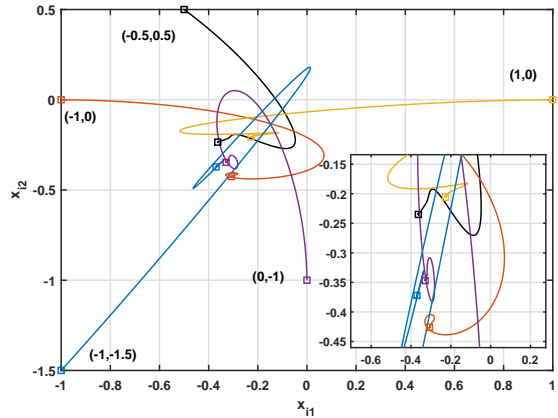}\\
  \caption{The trajectories of players' positions generated by \eqref{Y1}-\eqref{Y2}. }\label{12_tra}
\end{figure}
\begin{figure}
  \centering
  \includegraphics[scale=0.5]{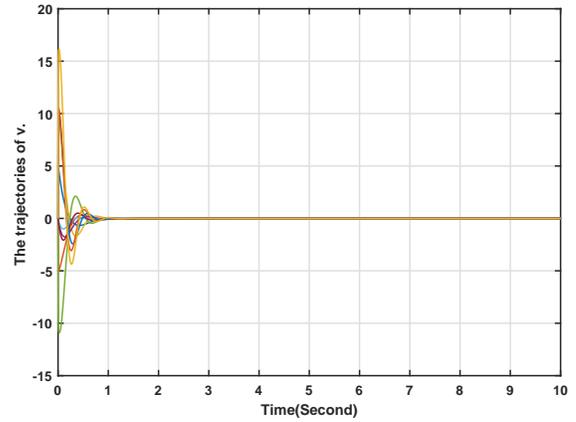}\\
  \caption{The players' velocities generated by \eqref{Y1}-\eqref{Y2}. }\label{12_tra1}
\end{figure}

\section{Conclusions}\label{conc}
This paper develops two Nash equilibrium strategies for games in which the players' actions are governed by second-order integrator-type dynamics. In particular, the players' velocities are supposed to be unavailable for feedback control of the players' positions. Without utilizing velocity measurement, an observer-based approach and a filter-based approach  are designed. Through Lyapunov stability analysis, it is theoretically shown that the players' positions and velocities would be steered to the Nash equilibrium and zero, respectively. Extensions to games in distributed networks are discussed. The presented results show that both the observer-based approach and the filter-based approach can be adapted to solve distributed games, thus showing their extensibility.  It would be interesting future works to extend the current work to the recently formulated $N$-cluster games (see \cite{8}-\cite{YE}) and non-model-based counterparts (see e.g., \cite{YESMAC}).


\begin{thebibliography}{99}
\bibitem{12}M. Ye, ``Distributed Nash equilibrium seeking for games in systems with bounded control inputs," submitted to \emph{IEEE Transactions on Automatic Control,} avaiable online at arXiv:1901.09333, 2019.
\bibitem{YEICCA19} M. Ye, ``Distributed strategy design for solving games in systems with bounded control inputs," \emph{IEEE International Conference on Control and Automation,} pp. 266-271, 2019.
\bibitem{11}A. Ibrahim, T. Hayakawa, ``Nash equilibrium seeking with second-order dynamic agents," \emph{IEEE Conference on Decision and Control,} pp. 2514-2518, 2018.
\bibitem{Y}J. Yin, M. Ye, ``Distributed Nash equilibrium computation for mixed-order multi-player games," submitted to \emph{IEEE Conference on Control and Automation,} 2020.
\bibitem{dMB}M. Bianchi, S. Grammatico, ``Continuous-time fully distributed generalized Nash equilibrium seeking for multi-integrator agents," arXiv preprint arXiv:1911.12266, 2019.
\bibitem{MB}M. Bianchi, S. Grammatico, ``A continuous-time distributed generalized Nash equilibrium seeking algorithm over networks for double-integrator agents," arXiv preprint arXiv:1910.11608, 2019.
\bibitem{Lim97} S. Lim, D. Dawson, J. Hu, and M. de Queiroz, ``An adaptive link position tracking controller for rigid-link flexible-joint robots without velocity measurements," \emph{IEEE Transactions on Systems, Man and Cybernetics: Part B: Cybernetics} vol. 27, no. 3, pp. 412-427, 1997.
\bibitem{Andreev17}A. Andreev, O. Peregudova, ``Stabilization of the preset motions of a holonomic mechanical systems without velocity measurement," \emph{Journal of Applied Mathematics and Mechanics,} vol. 81, pp. 95-105, 2017.
\bibitem{Abdessameud09}A. Abdessameud, and A. Tayebi, ``Attitude synchronization of a group of spacecraft without velocity measurements," \emph{IEEE Transactions on Automatic Control,}, vol. 54, no. 11, pp. 2642-2648, 2009.
\bibitem{DoTCST06}K. Do and J. Pan, ``Underactuated ships follow smooth paths with integral actions and without velocity measurements for feedback: theory and experiments," \emph{IEEE Transactions on Control Systems Technology,} vol. 14, no. 2, pp. 308-322, 2006.
\bibitem{Deng19} W. Deng, and J. Yao, ``Extended-state-observer-based adaptive control of electro-hydraulic servomechanisms without velocity measurement,"  \emph{IEEE/ASME Transactions on Mechatronics,} published online, DOI: 10.1109/TMECH.2019.2959297.
    \bibitem{MeiAT} J. Mei, W. Ren and G. Ma, ``Distributed coordination for second-order multi-agent systems with nonlinear dynamics using only relative position measurements," \emph{Automatica,} vol. 49, no. 5, pp. 1419-1427, 2013.
\bibitem{Zheng12}Y. Zheng, L. Wang, ``Finite-time consensus of heterogeneous multi-agent systems with and without velocity measurements," \emph{Systems and Control Letters,} vol. 61, no. 8, pp. 871-878, 2012.
\bibitem{Ren08} W. Ren and R. Beard, \emph{Distributed consensus in multi-vehicle cooperative control: Theory and Application,}  Springer London, 2008.
\bibitem{5}M. Ye, G. Hu, ``Distributed Nash equilibrium seeking by a consensus based approach," \emph{IEEE Transactions on Automatic Control,} vol. 62, no. 9, pp. 4811-4818, 2017.
\bibitem{6} M. Ye, G. Hu, ``Distributed Nash equilibrium seeking in multi-agent games under switching communication topologies," \emph{IEEE Transactions on Cybernetics,} vol. 48, no. 11, pp. 3208-3217, 2018.
\bibitem{YECYBER17}M. Ye, G. Hu, ``Game Design and Analysis for Price based Demand Response: An Aggregate Game Approach," \emph{IEEE Transactions on Cybernetics,} vol. 47, no. 3, pp. 720-730, 2017.
\bibitem{7} M. Ye, ``A RISE-based distributed robust Nash equilibrium seeking strategy for networked games," \emph{IEEE Conference on Decision and Control,} pp. 4047-4052, 2019.
%\bibitem{1}M. Ye, G. Hu, ``Game design and analysis for price-based demand response: an aggregate game approach," \emph{IEEE Transactions on Cybernetics,} vol. 47, no. 3, pp. 720 - 730, 2017.
%\bibitem{2}A. Cherukuri, J. Cort\'{e}s, ``Decentralized Nash equilibrium learning by strategic generators for economic dispatch," \emph{American Control Conference,} pp. 1082-1087, 2016.
%\bibitem{3}W. Lin, C. Li, Z. Qu, M. Simaan, ``Distributed formation control with open-loop Nash strategy," \emph{Automatica,} vol. 106, pp. 266-273, 2019.

%\bibitem{4} M. Ye, G. Hu, ``Game Design and Analysis for Price based Demand Response: An Aggregate Game Approach," \emph{IEEE Transactions on Cybernetics,} 47 (3), 720-730, 2017.
\bibitem{WANGSMCA}R. Wang, X. Dong, Q. Li and Z. Ren, ``Distributed time-varying formation control for linear swarm systems with switching topologies using an adaptive output-feedback approach," \emph{IEEE Transactions on Systems, Man, and Cybernetics: Systems,} vol. 49, no. 12, pp. 2664-2675, 2019.
\bibitem{WenSMACA} G. Wen, X. Yu, W. Yu, J. L¨¹, ``Coordination and Control of Complex Network Systems With Switching Topologies: A Survey," \emph{IEEE Transactions on Systems, Man, and Cybernetics: Systems,} published online, DOI: 10.1109/TSMC.2019.2961753.

\bibitem{8}M. Ye, G. Hu, and F. L. Lewis, ``Nash equilibrium seeking for n-coalition non-cooperative games," \emph{Automatica}, vol. 95, pp. 266-272, 2018.
\bibitem{9}M. Ye, G. Hu, F. L. Lewis, L. Xie, ``A unified strategy for solution seeking in graphical n-coalition noncooperative games," \emph{IEEE Transactions on Automatic Control,} vol. 64, no. 11, pp. 4645-4652, 2019.
\bibitem{YE}M. Ye, G. Hu, S. Xu, ``An extremum seeking-based approach for Nash equilibrium seeking in N-cluster noncooperative games," \emph{Automatica,} vol. 114, 108815, 2020.



%\bibitem{13}M. Ye, ``Nash equilibrium seeking for games in hybrid systems," \emph{2018 15th International Conference on Control, Automation, Robotics and Vision,} pp. 140-145, 2018.
%\bibitem{14}M. Zheng, C. Liu, F. Liu, ``Average-consensus tracking of sensor network via distributed coordination control of heterogeneous multi-agent systems," \emph{IEEE Control Systems Letters,} vol. 3, no. 1, pp. 132-137, 2018.
\bibitem{15}H. Khailil, \emph{Nonlinear Systems,} Upper Saddle River, NJ: Prentice Hall, 2002.

\bibitem{YESMAC}M. Ye, G. Wen, S. Xu, F. Lewis, ``Global social cost minimization with possibly nonconvex objective functions: an extremum seeking-based approach," \emph{IEEE Transactions on Systems, Man, and Cybernetics: Systems,} accepted, published online, DOI: 10.1109/TSMC.2020.2968959.
%\bibitem{16}J. Rosen, ``Existence and uniquness of equilibrium points for concave N-person games,`` \emph{Econometrica}, vol. 33, no. 3, p. 520, Jul. 1965.
%\bibitem{17}M. Ye, G. Hu, ``Adaptive approaches for fully distributed Nash equilibrium seeking in networked games," submitted to \emph{Automatica}, arXiv preprint arXiv:1912.00415, 2019.



\end{thebibliography}
\end{document}